\documentclass{amsart}
\usepackage{amsfonts}
\usepackage{amsmath}
\usepackage{amssymb}
\usepackage{amsthm}

\newtheorem{lemma}{Lemma}
\newtheorem{theorem}{Theorem}
\newtheorem{rem}{Remark}
\newtheorem{exmp}{Example}

\DeclareMathOperator{\GL}{GL}

\newcommand{\cG}{{\mathcal G}}
\newcommand{\cL}{{\mathcal L}}
\newcommand{\cM}{{\mathcal M}}
\newcommand{\cP}{{\mathcal P}}
\newcommand{\cX}{{\mathcal X}}
\newcommand{\dtimes}{\mathbin{\dot\times}}
\newcommand{\inz}{\mathrel{\mathrm{I}}}

\newcommand{\C}{C-{\hspace{0pt}}} 
\newcommand{\A}{A-{\hspace{0pt}}} 

\newenvironment{myproof}[1]%
    {\begin{trivlist} \item \emph{Proof of Theorem\/ }#1.}
    {{}\hfill $\square$\end{trivlist}}

\begin{document}

\title{Transformations on the product of Grassmann spaces}

\author{Hans Havlicek \and Mark Pankov}

\maketitle

\section{Introduction}\label{sect:intro}

Let $\cG_k$ denote the set of all $k$-dimensional subspaces of an
$n$-dimensional vector space. We recall that two elements of $\cG_k$ are called
\emph{adjacent} if their intersection has dimension $k-1$. The set $\cG_k$ is
point set of a partial linear space, namely a \emph{Grassmann space} for $1< k<
n-1$ (see Section \ref{sect:thm:GxG}) and a projective space for
$k\in\{1,n-1\}$. Two adjacent subspaces are---in the language of partial linear
spaces---two distinct collinear points.

W.L.~Chow \cite{Chow} determined all bijections of $\cG_k$ that preserve
adjacency in both directions in the year 1949. In this paper we call such a
mapping, for short, an \emph{{\A}transformation}. Disregarding the trivial
cases $k=1$ and $k=n-1$, every {\A}transformation of $\cG_k$ is induced by a
semilinear transformation $V\to V$ or (only when $k=2n$) by a semilinear
transformation of $V$ onto its dual space $V^*$. There is a wealth of related
results, and we refer to \cite{benz-92}, \cite{huang-98}, and \cite{wan-96} for
further references.

In the present note, we aim at generalizing Chow's result to products of
Grassmann spaces. However, we consider only products of the form $\cG_k\times
\cG_{n-k}$, where $\cG_k$ and $\cG_{n-k}$ stem from the same vector space $V$.
Furthermore, for a fixed $k$ we restrict our attention to a certain subset of
$\cG_k\times \cG_{n-k}$. This subset, say $\cG$, is formed by all pairs of
\emph{complementary\/} subspaces. Our definition of an adjacency on $\cG$ in
formula (\ref{eq:G.adjacent}) is motivated by the definition of lines in a
product of partial linear spaces; cf.\ e.g.\ \cite{NP}.

One of our main results (Theorem~\ref{thm:A}) states that Chow's theorem
remains true, mutatis mutandis, for the {\A}transformations of $\cG$. However,
in Theorem~\ref{thm:C} we can show even more: Let us say that two elements
$(S,U)$ and $(S',U')$ of $\cG$ are \emph{close\/} to each other, if their
Hamming distance is $1$ or, said differently, if they coincide in precisely one
of their components. Then the bijections of $\cG$ onto itself which preserve
this closeness relation in both directions---we call them
\emph{{\C}transformations\/} of $\cG$---are precisely the {\A}transformations
of $\cG$. In this way, we obtain for $1<k<n-1$ two characterizations of the
semilinear bijections $V\to V$ and $V\to V^*$ via their action on the set
$\cG$.

Finally, we turn to the following question: What happens to our results if we
replace the set $\cG$ with the entire cartesian product $\cG_k\times
\cG_{n-k}$? Clearly, the basic notions of adjacency and closeness remain
meaningful. We describe all {\C}transformations of $\cG_k\times \cG_{n-k}$ in
Theorem~\ref{thm:CGxG}. However, in sharp contrast to Theorem~\ref{thm:C}, this
is a rather trivial task, and the transformations of this kind do not deserve
any interest. Then, using a result of A.~Naumowicz and K.~Pra\.zmowski
\cite{NP}, we also determine all {\A}transformations of $\cG_k\times \cG_{n-k}$
in Theorem~\ref{thm:AGxG}. Such mappings are closely related with collineations
of the underlying partial linear space, and in general they can be described in
terms of \emph{two\/} semilinear bijections, but not in terms of a
\emph{single\/} semilinear bijection.

Before we close this section, it is worthwhile to mention that the results from
\cite{NP} could be used to describe the {\A}transformations of arbitrary finite
products of Grassmann spaces, but this is not the topic of the present article.

\section{{\A}transformations and {\C}transformations}\label{sect:trafos}

First, we collect our basic assumptions and definitions. Throughout this paper,
let $V$ be a $n$-dimensional left vector space over a division ring, $2\leq
n<\infty$. Suppose that $P,T\subset V$ are subspaces. They are said to be
\emph{incident} (in symbols: $P\inz T$) if $P\subset T$ or if $T\subset P $.
Note that according to this definition every subspace of $V$ is incident with
$0$ (the zero subspace) and with $V$. Furthermore, we define
\begin{equation}\label{eq:adjacent}
   P\sim T \;:\Leftrightarrow\; \dim P = \dim T = \dim (P\cap T) +1,
\end{equation}
where ``$\sim$'' is to be read as \emph{adjacent}.

We put ${\mathcal G}_{i}$, for the set $i$-dimensional subspaces of $V$,
$i=0,1,\dots, n$. In what follows \emph{we fix a natural number\/}
$k\in\{1,2,\ldots,n-1\}$ and adopt the notation
\begin{equation}\label{eq:G}
   \cG:=\{(S,U)\in{\mathcal G}_{k}\times {\mathcal G}_{n-k}\mid S+U=V\}.
\end{equation}
Hence $(S,U)\in\cG$ means that $S$ and $U$ are \emph{complementary\/}
subspaces. On the set $\cG$ we define two binary relations: Elements $(S,U)$
and $(S',U')$ of ${\mathcal G}$ are said to be \emph{adjacent\/} if
\begin{equation}\label{eq:G.adjacent}
   ( S=S'\mbox{ and } U\sim U')\mbox{ or }( S\sim S'\mbox{ and } U=U').
\end{equation}
By abuse of notation, this relation on $\cG$ will also be denoted by the symbol
``$\sim$''. Our elements are said to be \emph{close\/} to each other (in
symbols: $(S,U)\approx (S',U')$) if
\begin{equation}\label{eq:G.close}
   ( S=S'\mbox{ and } U\neq U')\mbox{ or }( S\neq S'\mbox{ and } U=U').
\end{equation}
According to this definition, any two adjacent elements of ${\mathcal G}$ are
close; the converse holds only for $k=1$ and $k=n-1$.

We shall establish in Lemma~\ref{lemma:connect} that any two elements $(S,U)$
and $(S',U')$ of ${\mathcal G}$ can be connected by a finite sequence
\begin{equation}\label{eq:A.connect}
  (S,U)=(S_{0},U_{0})\sim (S_{1},U_{1})\sim\cdots\sim (S_{i},U_{i})=(S',U').
\end{equation}
Consequently, we also have
\begin{equation}\label{eq:C.connect}
 (S,U)=(S_{0},U_{0})\approx (S_{1},U_{1})\approx\cdots\approx
 (S_{i},U_{i})=(S',U').
\end{equation}
We refer to the definition of a \emph{Pl\"ucker space\/} in
\cite[p.~199]{benz-92}, and we point out the (inessential) difference that our
relations $\sim$ and $\approx$ are anti-reflexive.

A bijection $f:{\mathcal G}\to {\mathcal G}$ is said to be an \emph{adjacency
preserving transformation} (shortly: an \emph{{\A}transformation}) if $f$ and
$f^{-1}$ transfer adjacent elements of $\cG$ to adjacent elements; if $f$ and
$f^{-1}$ map close elements of $\cG$ to close elements then we say that $f$ is
a \emph{closeness preserving transformation\/} (shortly: a
\emph{{\C}transformation}).

\begin{exmp}\label{exmp:1}{\rm
For any two mappings $f':{\mathcal G}_{k}\to {\mathcal G}_{k}$ and
$f'':{\mathcal G}_{n-k}\to {\mathcal G}_{n-k}$ we put
\begin{equation}
 f'\times f'': {\mathcal G}_{k}\times {\mathcal G}_{n-k}\to
               {\mathcal G}_{k}\times {\mathcal G}_{n-k} :
               (S,U)\mapsto \big(f'(S),f''(U)\big).
\end{equation}
Each semilinear isomorphism $l:V\to V$ induces, for $i=1,2,\dots,n-1$,
bijections
\begin{equation}
   G_{i}(l):{\mathcal G}_{i}\to {\mathcal G}_{i} : S\mapsto l(S).
\end{equation}
Obviously, the restriction of
\begin{equation}
   G_{k}(l)\times G_{n-k}(l)
\end{equation}
to ${\mathcal G}$ is an {\A}transformation and a {\C}transformation.
}\end{exmp}

\begin{exmp}\label{exmp:2}{\rm
For any two mappings $g':{\mathcal G}_{k}\to {\mathcal G}_{n-k}$ and
$g'':{\mathcal G}_{n-k}\to {\mathcal G}_{k}$ we put
\begin{equation}
   g'\dtimes g'':{\mathcal G}_{k}\times {\mathcal G}_{n-k}\to
                       {\mathcal G}_{k}\times {\mathcal G}_{n-k} :
  (S,U)\mapsto \big(g''(U),g'(S)\big).
\end{equation}
Let $V^*$ denote the dual space of $V$. Each semilinear isomorphism $s:V\to
V^{*}$ induces, for $i=1,2,\dots,n-1$, the bijections
\begin{equation}
    D_{i}(s):{\mathcal G}_{i}\to{\mathcal G}_{n-i} :
    S\mapsto \big(s(S)\big)^{\circ},
\end{equation}
where $\big(s(S)\big)^{\circ}$ denotes the annihilator of $s(S)$. The
restriction of
\begin{equation}
    D_{k}(s)\dtimes D_{n-k}(s)
\end{equation}
to ${\mathcal G}$ is an {\A}transformation and a {\C}transformation. Observe
that a necessary and sufficient condition for the existence of such an
isomorphism $s$ is that the underlying division ring admits an
anti-auto\-mor\-phism.}
\end{exmp}

\begin{exmp}\label{exmp:3}{\rm
Now suppose that $n=2k$. We assume that $l:V\to V$ and $s:V\to V^{*}$ are
semilinear isomorphisms. The restrictions of
\begin{equation}
      G_{k}(l)\dtimes G_{k}(l)\;\mbox{ and }\;D_{k}(s)\times D_{k}(s)
\end{equation}
to ${\mathcal G}$ both are {\A}transformations and {\C}transformations.
}\end{exmp}

\begin{exmp}\label{exmp:n=2}{\rm
Let $n=2$ and $k=1$. Choose an arbitrary bijection $f:\cG_1\to\cG_1$. Then the
restrictions of  $f\times f$ and $f\dtimes f$ to $\cG$ both are
{\A}transformations and {\C}transformations.}
\end{exmp}

We are now in a position to state our main results:

\begin{theorem}\label{thm:C}
Every closeness preserving transformation of ${\mathcal G}$ is one of the
mappings considered in Examples~{\rm\ref{exmp:1}--\ref{exmp:n=2}}. Hence it is
an adjacency preserving transformation.
\end{theorem}

It is trivial that each {\A}transformation is a {\C}transformation if $k=1$ or
if $k=n-1$. In Section~\ref{sect:thm:A} we shall prove this statement for the
general case. Thus the following statement holds true.

\begin{theorem}\label{thm:A}
Every adjacency preserving transformation of ${\mathcal G}$ is one of the
mappings considered in Examples~{\rm\ref{exmp:1}--\ref{exmp:n=2}}. Hence it is
a closeness preserving transformation.
\end{theorem}

It is clear that our definitions of adjacency and closeness remain meaningful
on the entire cartesian product $\cG_{k}\times\cG_{n-k}$. Also the notions of
{\C} and {\A}transformation and Examples~\ref{exmp:1}--\ref{exmp:n=2} can be
carried over accordingly. However, Theorems~\ref{thm:C} and \ref{thm:A} do not
remain unaltered when $\cG$ is replaced with $\cG_{k}\times\cG_{n-k}$:

\begin{exmp}\label{exmp:CGxG}{\rm
Let $f':\cG_{k}\to\cG_{k}$ and $f'':\cG_{n-k}\to\cG_{n-k}$ be bijections. Then
$f'\times f''$ is a {\C}transformation. Also, if $g':\cG_{k}\to\cG_{n-k}$ and
$g'':\cG_{n-k}\to\cG_{k}$ are bijections then $g'\dtimes g''$ is a
{\C}transformation.}
\end{exmp}

For the sake of completeness, let us state the following rather trivial result:

\begin{theorem}\label{thm:CGxG}
Every closeness preserving transformation of $\cG_{k}\times\cG_{n-k}$ is one of
the mappings considered in Example~{\rm\ref{exmp:CGxG}}.
\end{theorem}

\begin{exmp}\label{exmp:AGxG}{\rm
If $f':\cG_{k}\to\cG_{k}$ and $f'':\cG_{n-k}\to\cG_{n-k}$ are bijections which
preserve adjacency in both directions then $f'\times f''$ is an
{\A}transformation. Also, if $g':\cG_{k}\to\cG_{n-k}$ and
$g'':\cG_{n-k}\to\cG_{k}$ are bijections which preserve adjacency in both
directions then $g'\dtimes g''$ is an {\A}transformation.

Suppose that $k=1$ or $k=n-1$. Then it suffices to require that the mappings
$f'$, $f''$, $g'$ and $g''$ from above are bijections in order to obtain an
{\A}transformation of $\cG_{k}\times\cG_{n-k}$.

Provided that $1<k<n-1$, we can apply Chow's theorem (\cite[p.~38]{Chow},
\cite[p.~81]{dieu-71}) to describe explicitly the mappings from above.

In the first case we have $f'=G_{k}(l')$ or $f'=D_{k}(s')$ (only when $n=2k$),
and $f''=G_{n-k}(l'')$ or $f''=D_{k}(s'')$ (only when $n=2k$).

In the second case we have $g'=D_{k}(s')$ or $g'=G_{k}(l')$ (only when $n=2k$),
and $g''=D_{n-k}(s'')$ or $g''=G_{k}(l'')$ (only when $n=2k$).

Here $l',l'':V\to V$ and $s',s'':V\to V^*$ denote semilinear isomorphisms.

 }\end{exmp}

We shall see that the following result is a consequence of
\cite[Theorem~1.14]{NP}:

\begin{theorem}\label{thm:AGxG}
Every adjacency preserving transformation of $\cG_{k}\times\cG_{n-k}$ is one of
the mappings considered in Example~{\rm\ref{exmp:AGxG}}.
\end{theorem}

\begin{rem}{\rm
Suppose that the underlying division ring of $V$ is not of characteristic $2$.
Let $u\in \GL(V)$ be an involution. Then there exist two invariant subspaces
$U_{+}(u)$ and $U_{-}(u)$ with $V=U_{+}(u)\oplus U_{-}(u)$ such that $u(x)=\pm
x$ for each $x\in U_{\pm}(u)$. If $\dim U_{+}(u)= r$ then $\dim U_{-}(u)=n-r$, and
$u$ is called an \emph{$(r,n-r)$-involution}.

For our fixed $k$ let $J$ be the set of all $(k,n-k)$-involutions. There exists
a bijection
\begin{equation}\label{}
      \gamma: J\to \cG : u \mapsto \big(U_{+}(u),U_{-}(u)\big).
\end{equation}
Two $(k,n-k)$-involutions $u$ and $v$ are said to be \emph{adjacent\/} if the
corresponding elements of ${\mathcal G}$ are adjacent. This holds if, and only
if, the product of $u$ and $v$ (in any order) is a transvection $\neq 1_V$.

Now let $f:J\to J$ be a bijection which preserves adjacency in both directions.
We apply Theorem~\ref{thm:A} to the {\A}transformation $\gamma f\gamma^{-1}:
{\mathcal G}\to\cG$. If $n>2$ and $n\ne 2k$ then this last mapping is given as
in Example~\ref{exmp:1} or \ref{exmp:2}. This means that  $f$ can be extended
to an automorphism of the group $\GL(V)$ as follows: To each $u\in \GL(V)$ we
assign $lul^{-1}$ or the contragredient  of $sus^{-1}$, respectively.}
\end{rem}

\section{Proof of Theorem~\ref{thm:C}}\label{sect:thm:C}

Our proof of Theorem~\ref{thm:C} will be based on several lemmas and the
subsequent characterization. In the case $n=2k$ this statement is a particular
case of a result in \cite{BlunckHavlicek}. The direct analogue of
Theorem~\ref{thm:3} for buildings can be found in \cite[Proposition~4.2]{AVM}.

\begin{theorem}\label{thm:3}
Let $1\leq k\leq n-1$. Then for any two distinct $S_{1},S_{2}\in {\mathcal
G}_{k}$ the following two conditions are equivalent:
\begin{enumerate}
  \item[(a)] $S_{1}$ and $S_{2}$ are adjacent,
  \item[(b)]
There exists an $S\in {\mathcal G}_{k}-\{S_{1},S_{2}\}$ such that for all
$U\in {\mathcal G}_{n-k}$ the condition $(S,U)\in {\mathcal G}$ implies that
$(S_{1},U)$ or $(S_{2},U)$ belongs to ${\mathcal G}$.
\end{enumerate}
\end{theorem}

\begin{proof}
(a) $\Rightarrow$ (b). If $S_{1}$ and $S_{2}$ are adjacent then $S_{1}\cap
S_{2}\in{\mathcal G}_{k-1}$ and $S_{1}+S_{2}\in{\mathcal G}_{k+1}$. Every
$S\in {\mathcal G}_{k}-\{S_{1},S_{2}\}$ satisfying the condition
\begin{equation}
  S_{1}\cap S_{2}\subset S\subset S_{1}+S_{2}
\end{equation}
has the required property, and at least one such $S$ exists.

(b) $\Rightarrow$ (a). The proof of this implication will be given in several
steps. First we show that
\begin{equation}\label{eq:W-S}
  0\neq W_{1}\subset S_1\mbox{ and } 0\neq W_{2}\subset S_{2} \Rightarrow
  (W_{1}+W_{2})\cap S\neq 0.
\end{equation}
Assume, contrary to (\ref{eq:W-S}), that $(W_{1}+W_{2})\cap S=0$. Then there
exists a complement $U\in {\mathcal G}_{n-k}$ of $S$ containing
$W_{1}+W_{2}$. By our hypothesis, $U$ is a complement of $S_{1}$ or $S_{2}$.
This contradicts $W_{1}\subset S_{1}$ and $W_{2}\subset S_{2}$.

Our second assertion is
\begin{equation}\label{eq:S1-S2-S}
      S_{1}\cap S_{2}\subset S.
\end{equation}
This inclusion is trivial if $S_{1}\cap S_{2}$ is zero. Otherwise, let
$P\subset S_{1}\cap S_{2}$ be an arbitrarily chosen $1$-dimensional subspace.
We apply (\ref{eq:W-S}) to $W_{1}=W_{2}=P$. This shows that $P\cap S\ne 0$.
Hence $P\subset S$, as required.

The third step is to show that
\begin{equation}\label{eq:dim.Si-S}
     \dim (S\cap S_{1}) =\dim (S\cap S_{2}) = k-1.
\end{equation}
By symmetry, it suffices to establish that
\begin{equation}
     W_1\cap(S\cap S_{1})\ne 0
\end{equation}
for all $2$-dimensional subspaces $W_1\subset S_{1}$: Let us take a
$1$-dimensional subspace $P_{2}\subset S_{2}$ such that $P_{2}\cap S=0$. Then
(\ref{eq:S1-S2-S}) implies that $P_{2}$ is not contained in $S_{1}$, and for
every $2$-dimensional subspace $W_1\subset S_{1}$ the subspace $W_1+P_{2}$ is
$3$-dimensional. Let $P_{1}$ and $Q_{1}$ be distinct $1$-dimensional
subspaces contained in $W_1$. It follows from (\ref{eq:W-S}) that
$P_{1}+P_{2}$ and $Q_{1}+P_{2}$ meet $S$ in $1$-dimensional subspaces ($\neq
P_{2}$) which will be denoted by $P$ and $Q$, respectively. As $P_{1}$ and
$Q_{1}$ are distinct, so are $P$ and $Q$. Therefore $P+Q$ is a
$2$-dimensional subspace of $S$. Since $W_1$ and $P+Q$ lie in the
$3$-dimensional subspace $W_1+P_{2}$, they have a common $1$-dimensional
subspace contained in $W_1\cap S = W_1\cap(S\cap S_{1})$. This proves
(\ref{eq:dim.Si-S}).

Finally, we read off from (\ref{eq:S1-S2-S}) that
\begin{equation}\label{eq:S1-S2}
      S_{1}\cap S_{2}=(S\cap S_{1})\cap (S\cap S_{2}),
\end{equation}
and we shall finish the proof by showing that this subspace has dimension
$k-1$. By (\ref{eq:dim.Si-S}) and because of $S_1\neq S_2$, the dimension of
$S_{1}\cap S_{2}$ is either $k-2$ or $k-1$. Suppose, to the contrary, that
\begin{equation}\label{eq:k-2}
      \dim S_{1}\cap S_{2}=k-2.
\end{equation}
Then $S\cap S_{1}$ and $S\cap S_{2}$ are distinct $(k-1)$-dimensional
subspaces spanning $S$.
There exist $1$-dimensional subspaces
$P_{1},P_{2}$ such that
\begin{equation}
     S_{i}=(S\cap S_{i})+P_{i}
\end{equation}
for $i=1,2$. We have $P_{1}\ne P_{2}$ (otherwise (\ref{eq:S1-S2-S}) would
give $P_{1}=P_{2}\subset S_{1}\cap S_{2}\subset S$ which is impossible), and
(\ref{eq:W-S}) guarantees that $(P_{1}+P_{2})\cap S$ is a $1$-dimensional
subspace. Then $S_{1}+S_{2}$ is contained in the $(k+1)$-dimensional subspace
$S+P_1$ which, by the dimension formula for subspaces, contradicts
(\ref{eq:k-2}).
\end{proof}

\begin{lemma}\label{lemma:1}
If $l:V\to V$ is a semilinear isomorphism such that $G_{j}(l)$ is the identity
for at least one $j\in\{1,2,\ldots,n-1\}$ then the same holds for all
$i=1,2\ldots, n-1$.
\end{lemma}
\begin{proof}
This is well known.
\end{proof}

\begin{lemma}\label{lemma:2}
Let $l_{i}:V\to V$ and $s_{i}:V\to V^{*}$ be semilinear isomorphisms, $i=1,2$.
Then the following assertions hold.
\begin{enumerate}

\item[(a)] If one of the mappings $G_{k}(l_{1})\times G_{n-k}(l_{2})$ or
$G_{k}(l_{1})\dtimes G_{k}(l_{2})$, when restricted to $\cG$, is a
{\C}transformation then $G_{i}(l_{1})=G_{i}(l_{2})$ for all $i=1,2,\dots, n-1$.

\item[(b)] If one of the mappings $D_{k}(s_{1})\dtimes D_{n-k}(s_{2})$ or
$D_{k}(s_{1})\times D_{k}(s_{2})$, when restricted to $\cG$, is a
{\C}transformation then $D_{i}(s_{1})=D_{i}(s_{2})$ for all $i=1,2,\dots, n-1$.

\item[(c)] If $n=2k>2$ then none of the mappings
  $G_{k}(l_{1})\times D_{k}(s_{2})$,
  $D_{k}(s_{1})\times G_{k}(l_{2})$,
  $G_{k}(l_{1})\dtimes D_{k}(s_{2})$,  and
  $D_{k}(s_{1})\dtimes G_{k}(l_{2})$
is a {\C}transformation, when it is restricted to $\cG$.
\end{enumerate}
\end{lemma}

\begin{proof}
(a) Let the restriction of $G_{k}(l_{1})\times G_{n-k}(l_{2})$ to $\cG$ be a
{\C}transformation. Then $G_{k}(1_{V})\times G_{n-k}(l^{-1}_{1}l_{2})$ gives
also a {\C}transformation. This means that for each $U\in {\mathcal G}_{n-k}$
the mapping $G_{k}(1_{V})$ transfers the set of all $k$-dimensional subspaces
having a non-zero intersection with $U$ onto the set of all $k$-dimensional
subspaces having a non-zero intersection with $l^{-1}_{1}l_{2}(U)$. However,
$G_{k}(1_{V})$ is the identity. Thus
\begin{equation}
  l^{-1}_{1}l_{2}(U)=U,
\end{equation}
and $G_{n-k}(l_{2}l^{-1}_{1})$ is the identity. Hence we can apply
Lemma~\ref{lemma:1} to show the assertion in this particular case.

Next, let the restriction of $G_{k}(l_{1})\dtimes G_{k}(l_{2})$ to $\cG$ be a
{\C}transformation. Thus $n=2k$ and the assertion follows from the previous
case and
\begin{equation}
  G_{k}(l_{1})\dtimes G_{k}(l_{2}) = \big(G_{k}(1_V)\dtimes G_{k}(1_V)\big)
  \big(G_{k}(l_{1})\times G_{k}(l_{2})\big).
\end{equation}

(b) can be verified similarly to (a).

(c) Assume, contrary to our hypothesis, that $G_{k}(l_{1})\times D_{k}(s_{2})$
gives a {\C}transformation. Hence $G_{k}(1_V)\times D_{k}(s_{2}l_1^{-1})$ is
also a {\C}transformation and, as above, we infer that
\begin{equation}
   D_{k}(s_{2}l_1^{-1})(U)=\big((s_{2}l_1^{-1})(U)\big)^\circ = U
\end{equation}
for all $U\in\cG_{k}$. Let $W\in\cG_{k-1}$. Then there are subspaces
$U_1,U_2,\ldots U_{k+1}\in\cG_{k}$ such that  $V = \sum_{i=1}^{k+1}U_i$ and  $W
= \bigcap_{i=1}^{k+1}U_i$. Consequently,
\begin{equation}
   0=\big(s_{2}l_1^{-1}(V)\big)^\circ
    =\bigcap_{i=1}^{k+1}\big((s_{2}l_1^{-1})(U_i)\big)^\circ
    =\bigcap_{i=1}^{k+1} U_i
    = W
\end{equation}
which implies $k=1$, an absurdity.

The remaining cases can be shown in the same way.
\end{proof}
Let us remark that in general the assumption $n>2$ in part (c) of this lemma
cannot be dropped. Indeed, if $n=2k=2$ and if $K$ is a commutative field then
there exists a non-degenerate alternating bilinear form $b:V\times V\to K$.
Hence $s:V\to V^*: v\mapsto b(v,\cdot)$ is a linear bijection, and
$G_1(1_V)\times D_1(s)$ is the identity on $\cG_1\times\cG_1$.

\begin{lemma}\label{lemma:n=2}
Let $n=2$, whence $k=1$. Suppose that $g':\cG_1\to \cG_1$ and
$g'':\cG_1\to\cG_1$ are bijections such that one of the mappings $g'\times g''$
or $g'\dtimes g''$, when restricted to $\cG$, is a {\C}transformation. Then
$g'=g''$.
\end{lemma}

\begin{proof}
It suffices to discuss the first case, since $1_\cG\dtimes 1_\cG$ yields a
{\C}transformation. Now we can proceed as in the proof of Lemma~\ref{lemma:2}
(a) in order to establish that the restriction of $g'{}^{-1}g''$ to $\cG$
equals $1_\cG$.
\end{proof}

We say that ${\mathcal X}\subset{\mathcal G}$ is a \emph{{\C}subset}  if any
two distinct elements of ${\mathcal X}$ are close. (If we consider the graph of
the closeness relation on $\mathcal G$ then a {\C}subset is just a clique, i.e.
a complete subgraph.) A {\C}subset is said to be \emph{maximal\/} if it is not
properly contained in any {\C}subset. In order to describe the maximal
{\C}subsets the following notation will be useful. If $P$ and $T$ are subspaces
of $V$ then we put
\begin{equation}\label{}
    \cG(P,T):=\{(S,U)\in\cG\mid S\inz P \mbox{ and } U\inz T\};
\end{equation}
here we use the incidence relation from the beginning of
Section~\ref{sect:trafos}.

\begin{lemma}\label{lemma:3}
The maximal {\C}subsets of ${\mathcal G}$ are precisely the sets $\cG(S,V)$
with $S\in\cG_{k}$, and $\cG(V,U)$ with $U\in\cG_{n-k}$.
\end{lemma}

\begin{proof}
Easy verification.
\end{proof}
We refer to the sets described in the lemma as maximal {\C}subsets of
\emph{first kind} and \emph{second kind}, respectively.
\begin{myproof}{\ref{thm:C}}
(a) Let $f$ be a {\C}transformation of ${\mathcal G}$. Then $f$ and $f^{-1}$
map maximal {\C}subsets to maximal {\C}subsets. Observe that two maximal
{\C}subsets have a unique common element if, and only if, one of them is of
first kind, say $\cG(S,V)$, the other is of second kind, say $\cG(V,U)$, and
$(S,U)\in\cG$.

Given $S,S'\in\cG_{k}$ there exists a subspace $U\in\cG_{n-k}$ such that
$S+U=S'+U=V$. We conclude from
\begin{equation}
  f\big(\cG(S,V)\big)\cap f\big(\cG(V,U)\big)=\{f\big((S,U)\big)\}
\end{equation}
that $f\big(\cG(S,V)\big)$ and $f\big(\cG(V,U)\big)$ are maximal {\C}subsets of
different kind. Likewise, $f\big(\cG(S',V)\big)$ and $f\big(\cG(V,U)\big)$ are
of different kind, so that $f\big(\cG(S,V)\big)$ and $f\big(\cG(S',V)\big)$ are
of the same kind.

A similar argument holds for maximal {\C}subsets of second kind; altogether the
action of the {\C}transformation $f$ on the set of maximal {\C}subsets is
either \emph{type preserving\/} or \emph{type interchanging\/}.

(b) Suppose that $f$ is type preserving. Then there exist bijections
\begin{eqnarray*}\label{}
    &g':\cG_{k} \to \cG_{k} \mbox{ such that } f\big(\cG(S,V)\big) = \cG\big(g'(S),V\big)
    \mbox{ for all }S\in\cG_{k},&\\
    &g'':\cG_{n-k} \to \cG_{n-k} \mbox{ such that }
    f\big(\cG(V,U)\big) = \cG\big(V,g''(U)\big)\mbox{ for all }U\in\cG_{n-k};&
\end{eqnarray*}
thus $f$ equals the restriction of $g'\times g''$ to $\cG$. We distinguish four
cases:

Case 1: $n=2$. Hence $k=1$; we deduce from Lemma~\ref{lemma:n=2} (a) that
$g'=g''$, whence $f$ is given as in Example~\ref{exmp:n=2}.

Case 2: $n>2$ and $k=1$.  Then for each $U\in {\mathcal G}_{n-1}$ the mapping
$g'$ transfers the set of all $1$-dimensional subspaces contained in $U$ to the
set of all $1$-dimensional subspaces contained in $g''(U)$.  This means, by the
fundamental theorem of projective geometry, that there exists a semilinear
isomorphism $l':V\to V$ with $g'=G_1(l')$. Similarly, $g''$ is induced by a
semilinear isomorphism $l'':V\to V$.

Case 3: $n>2$ and $k=n-1$. By symmetry, this coincides with the previous case.

Case 4: $n>2$ and $1<k<n-1$. Then Theorem~\ref{thm:3} guarantees that $g'$ and
$g''$ are adjacency preserving in both directions; Chow's theorem
(\cite[p.~38]{Chow}, \cite[p.~81]{dieu-71}) says that $g'$ and $g''$ are
induced by semilinear isomorphisms. More precisely, we have $g'=G_{k}(l')$ with
a semilinear bijection $l':V\to V$, or $g'=D_{k}(s')$ with a semilinear
bijection $s':V\to V^*$ (only when $n=2k$). A similar description holds for
$g''$.

In cases 2--4 we infer from Lemma~\ref{lemma:2} (c) that there are only two
possibilities:

Case A. $g'=G_{k}(l')$ and $g''=G_{n-k}(l'')$. Now Lemma~\ref{lemma:2} (a)
yields that $G_i(l')=G_i(l'')$ for all $i=1,2\ldots,n-1$, whence $f$ is the
restriction to $\cG$ of $G_{k}(l')\times G_{n-k}(l')$; cf.\
Example~\ref{exmp:1}.

Case B. $n=2k$, $g'=D_{k}(s')$, and $g''=D_{k}(s'')$. Now Lemma~\ref{lemma:2}
(b) yields that $D_i(s')=D_i(s'')$ for all $i=1,2\ldots,n-1$, whence $f$ is the
restriction to $\cG$ of $D_{k}(s')\times D_{k}(s')$; cf.\ Example~\ref{exmp:3}.

(c) If $f$ is type interchanging then there exist bijections
\begin{eqnarray*}\label{}
    &g':\cG_{k} \to \cG_{n-k} \mbox{ such that } f\big(\cG(S,V)\big) = \cG\big(V,g'(S)\big)
    \mbox{ for all }S\in\cG_{k},&\\
    &g'':\cG_{n-k} \to \cG_{k} \mbox{ such that }
    f\big(\cG(V,U)\big) = \cG(g''(U),V)\mbox{ for all }U\in\cG_{n-k};&
\end{eqnarray*}
thus $f$ is the restriction to $\cG$ of $g'\dtimes g''$. Now we can proceed,
mutatis mutandis, as in (b). So $f$ is given as in Example~\ref{exmp:n=2},
\ref{exmp:2}, or \ref{exmp:3}.

This completes the proof.
\end{myproof}

\section{Proof of Theorem~\ref{thm:A}}\label{sect:thm:A}

First, let us introduce the following notion: We say that ${\mathcal
X}\subset{\mathcal G}$ is an \emph{{\A}subset}  if any two distinct elements of
${\mathcal X}$ are adjacent. (As before, such a set is just a clique of the
graph given by the adjacency relation on $\mathcal G$.) An {\A}subset is said
to be \emph{maximal} if it is not properly contained in any {\A}subset.

If $k=1$ or if $k=n-1$ then an {\A}subset is the same as a {\C}subset, and
Lemma~\ref{lemma:3} can be applied.

\begin{lemma}\label{lemma:5}
Let $1< k < n-1$. Then the maximal {\A}subsets of ${\mathcal G}$ are precisely
the following sets:
 {\arraycolsep0pt\begin{eqnarray}\label{eq:typeA1}
  \cG(S,T)
  \mbox{ with }&S\in\cG_{k}\mbox{, }T\in\cG_{n-k+1}\mbox{, and }&S+T =V.\\
  \label{eq:typeA2}
  \cG(S,T)
  \mbox{ with }&S\in\cG_{k}\mbox{, }T\in\cG_{n-k-1}\mbox{, and }&S\cap T =0.\\
    \label{eq:typeA3}
  \cG(T,U)
  \mbox{ with }&T\in\cG_{k+1}\mbox{, }U\in\cG_{n-k}\mbox{, and }&T+U =V.\\
  \label{eq:typeA4}
  \cG(T,U)
  \mbox{ with }&T\in\cG_{k-1}\mbox{, }U\in\cG_{n-k}\mbox{, and }&T\cap U =0.
\end{eqnarray}}
\end{lemma}

\begin{proof}
From \cite[p.~36]{Chow} we recall the following: Let ${\mathcal Y}\subset
{\mathcal G}_{i}$, $1<i<n-1$, be a maximal set of mutually adjacent
$i$-dimensional subspaces of $V$. Then there exists a subspace $T\in {\mathcal
G}_{i\pm 1}$ such that ${\mathcal Y}=\{Y\in\cG_i\mid Y\inz T\}$.

Suppose now that $\cX\subset\cG$ is a maximal {\A}subset. Clearly, there exists
an element $(S,U)\in\cX$. Since $\cX$ is also a {\C}subset, we obtain that
$\cX\subset\cG(S,V)$ or that $\cX\subset\cG(V,U)$.

Let $\cX\subset\cG(S,V)$. Then the second components of the elements of $\cX$
are mutually adjacent elements of $\cG_{n-k}$. Hence, by the above, they all
are incident with a subset $T\in\cG_{n-k\pm 1}$.  So, due to its maximality,
the set $\cX$ is given as in (\ref{eq:typeA1}) or (\ref{eq:typeA2}).

Similarly, if $\cX\subset\cG(V,U)$ then $\cX$ can be written as in
(\ref{eq:typeA3}) or (\ref{eq:typeA4}).

Conversely, it is obvious that (\ref{eq:typeA1})--(\ref{eq:typeA4}) define
maximal {\A}subsets.
\end{proof}

We shall also make use of the following result:

\begin{lemma}\label{lemma:connect}
Any two elements $(S,U)$ and $(S',U')$ of ${\mathcal G}$ can be connected by a
finite sequence which is given as in formula {\rm(\ref{eq:A.connect})}. In
particular, if $S=S'$ {\rm(}or $U=U'${\rm)} then this sequence can be chosen in
such a way that $S=S_0=S_1=\cdots=S_i$ {\rm(}or $U=U_0=U_1=\cdots=U_i${\rm)}.
\end{lemma}

\begin{proof}
(a) First, we show the particular case when $(S,U),(S,U')\in\cG(S,V)$ with
$S\in\cG_k$. We proceed by induction on $d:=(n-k)-\dim(U\cap U')$, the case
$d=0$ being trivial.

Let $d>0$. There exists an $(n-k-1)$-dimensional subspace $W$ such that $U\cap
U'\subset W\subset U$. So $H:=W\oplus S$ is a hyperplane of $V$. It cannot
contain $U'$ because of $(S,U')\in\cG$. Thus $W' := H\cap U'$ has dimension
$n-k-1$, and there exists a $1$-dimensional subspace $P'\subset U'$ with
$U'=P'\oplus W'$. Consequently, $P'\not\subset H$ and we obtain
\begin{equation}\label{}
    V = P'\oplus H = P'\oplus W\oplus S.
\end{equation}
This means that $U'':=P'\oplus W$ is a complement of $S$. We have $(S,U)\sim
(S,U'')$ and $(n-k)-\dim(U''\cap U')=d-1$. So the assertion follows from the
induction hypothesis, applied to $(S,U'')$ and $(S,U')$.

Similarly, any two elements of ${\mathcal G}(V,U)$ with $U\in {\mathcal
G}_{n-k}$ can be connected.

(b) Now we consider the general case. Let $(S,U)$ and $(S',U')$ be elements of
${\mathcal G}$. There exists $U''\in {\mathcal G}_{n-k}$ which is complementary
to both $S$ and $S'$. Then, by (a), there exists a sequence
\begin{equation}\label{}
      (S,U)\sim\cdots\sim (S,U'')\sim\cdots\sim(S',U'')\sim\cdots\sim(S',U')
\end{equation}
which completes the proof.
\end{proof}

The statement in (a) from the above is just a particular case of a more general
result on the connectedness of a \emph{spine space}; cf.\
\cite[Proposition~2.9]{praz+z-02}.

\begin{myproof}{\ref{thm:A}} (a) We shall accomplish our task by showing that every {\A}transformation is a
{\C}transformation. As has been noticed in Section~\ref{sect:trafos}, this is
trivial if $k=1$ or if $k=n-1$. So let $f$ be an {\A}transformation of
${\mathcal G}$ and assume that $1<k<n-1$.

(b) We claim that
\begin{equation}\label{eq:Cinvariant}
    f\big(\cG(S,V)) \mbox{ is a maximal {\C}subset for all }S\in{\mathcal G}_{k}.
\end{equation}

Let us take $T\in {\mathcal G}_{n-k+1}$ such that ${\mathcal G}(S,T)$ is a
maximal {\A}subset. Then $f\big({\mathcal G}(S,T)\big)$ is also a maximal
{\A}subset. According to Lemma~\ref{lemma:5} there are four possible cases.

Case 1: $f\big({\mathcal G}(S,T)\big)$ is given according to (\ref{eq:typeA1}).
This means $f\big({\mathcal G}(S,T)\big)={\mathcal G}(W,Z)$ with $W\in\cG_{k}$,
$Z\in\cG_{n-k+1}$, and $W+Z=V$. We assert that in this case
\begin{equation}\label{eq:Ctrafo}
         f\big((S,U')\big)\in\cG(W,V) \mbox{ for all } (S,U')\in\cG(S,V).
\end{equation}
In order to show this we choose an element $(S,U)\in\cG(S,T)$. Clearly,
$f\big((S,U)\big)\in\cG(W,Z)\subset\cG(W,V)$.

First, we suppose that $(S,U)$ and $(S,U')$ are adjacent. Then $P:=U\cap
U'\in\cG_{n-k-1}$. We consider the \emph{pencil\/} given by $P$ and $T$, i.e.
the set
\begin{equation}
  \{X\in\cG_{n-k}\mid P\subset X\subset T\}.
\end{equation}
It contains at least three elements; precisely one them is not complementary to
$S$. Consequently, the intersection of the maximal {\A}subsets ${\mathcal
G}(S,T)$ and ${\mathcal G}(S,P)$ contains more than one element. The same
property holds for the intersection of the maximal {\A}subsets $f\big({\mathcal
G}(S,T)\big)=\cG(W,Z)$ and $f\big({\mathcal G}(S,P)\big)$. But this means that
$W$ is the first component of every element of $f\big({\mathcal G}(S,P)\big)$
so that $ f\big((S,U')\big)\in\cG(W,V)$.

Next, we suppose that $(S,U)$ and $(S,U')$ are arbitrary. By
Lemma~\ref{lemma:connect}, $(S,U)$ and $(S,U')$ can be connected by a finite
sequence
\begin{equation} (S,U)=(S,U_{0})\sim (S,U_{1})\sim\dots\sim (S,U_{i})=(S,U'),
\end{equation}
and the arguments considered above yield that (\ref{eq:Ctrafo}) holds.

Since $f^{-1}$ is adjacency preserving, we can repeat our previous proof, with
$\cG(W,Z)$ taking over the role of $\cG(S,T)$. Altogether, this proves
\begin{equation}\label{eq:}
  f\big({\mathcal G}(S,V)\big)=\cG(W,V).
\end{equation}

The remaining cases, i.e., when $f\big({\mathcal G}(S,T)\big)$ is given
according to (\ref{eq:typeA2}), (\ref{eq:typeA3}), or (\ref{eq:typeA4}), can be
treated similarly, whence (\ref{eq:Cinvariant}) holds true.

(c) Dual to (b), it can be shown that $f\big(\cG(V,U)\big)$ is a maximal
{\C}subset for all $U\in\cG_{n-k}$. Thus $f$ is a {\C}transformation.
\end{myproof}

\section{Proofs of Theorem~\ref{thm:CGxG} and Theorem~\ref{thm:AGxG}}
\label{sect:thm:GxG}

In the following proof we use the term \emph{maximal {\C}subset} just like in
Section~\ref{sect:thm:C}.

\begin{myproof}{\ref{thm:CGxG}}
Obviously, each maximal {\C}subset of $\cG_{k}\times\cG_{n-k}$ has either the
form $\{S\}\times\cG_{n-k}$ with $S\in\cG_k$ (\emph{first kind}) or
$\cG_k\times\{U\}$ with $U\in\cG_{n-k}$ (\emph{second kind}). Distinct maximal
{\C}subsets of the same kind have empty intersection, whereas maximal
{\C}subsets of different kind have a unique common element. So every
{\C}transformation is either type preserving, whence it can be written as
$f'\times f''$, or type interchanging, whence it can be written as $g'\dtimes
g''$.
\end{myproof}

Let $1<k<n-1$. We shall consider below the following well known \emph{partial
linear spaces}: For each $i=2,3,\ldots,n-2$ the set $\cG_i$ is the point set of
the \emph{Grassmann space\/} $(\cG_i,\cL_i)$; the elements of its line set
$\cL_i$ are the pencils
\begin{equation}
   \cG_i[P,T]:=\{X\in\cG_i\mid P\subset X\subset T\},
\end{equation}
where $P\in\cG_{i-1}$, $T\in\cG_{i+1}$, and $P\subset T$. The \emph{Segre
product\/} (or \emph{product space\/}) of $(\cG_{k},\cL_{k})$ and
$(\cG_{n-k},\cL_{n-k})$ is the partial linear space with point set
\begin{equation}\label{}
\cP:=\cG_{k}\times\cG_{n-k}
\end{equation}
and line set
\begin{equation}\label{}
    \cL:=\big\{\{S\}\times l \mid S\in\cG_{k},\, l\in\cL_{n-k} \big\}
    \cup
         \big\{m\times \{U\} \mid m\in\cL_{k},\, U\in\cG_{n-k} \big\}.
\end{equation}
See \cite{NP} for further details and references.

\begin{myproof}{\ref{thm:AGxG}}

(a) If $k=1$ or if $k=n-1$ then the assertion follows from
Theorem~\ref{thm:CGxG}.

(b) Let $1<k<n-1$. Given a subset $\cM\subset \cP$ we put
\begin{equation}\label{}
    \cM^\perp:=\{(S,U)\in\cP \mid (S,U)\perp (X,Y) \mbox{ for
    all }(X,Y)\in\cM\},
\end{equation}
where the sign ``$\perp$'' on the right hand side means ``adjacent or equal''.
Now let $(S,U)$ and $(S,U')$ be adjacent elements of $\cP $. Then
\begin{equation}
    \{(S,U),(S,U')\}^\perp =
    \{(S,Y)\in\cP \mid  U\cap U'\subset Y
    \mbox{ or } Y\subset U+U' \}
\end{equation}
and
\begin{equation}\label{eq:line1}
    \{(S,U),(S,U')\}^\perp{}^\perp =
    \{(S,Y)\in\cP \mid  U\cap U'\subset Y\subset U+U' \}.
\end{equation}
Similarly, if $(S,U)$ and $(S',U)$ are adjacent elements of $\cP $ then
\begin{equation}\label{eq:line2}
    \{(S,U),(S',U)\}^\perp{}^\perp =
    \{(X,U)\in \cP  \mid S\cap S'\subset X\subset S+S' \}.
\end{equation}

Next, suppose that $g:\cP\to\cP$ is an {\A}transformation. Every line of
$(\cP,\cL)$ can be written in the form (\ref{eq:line1}) or (\ref{eq:line2}),
since it contains at least two distinct collinear points or, said differently,
two adjacent elements of $\cP$. Thus $g$ is a collineation of the product space
$(\cP,\cL)$. By \cite[Theorem~1.14]{NP}, there are two possibilities:

Case 1. There exist collineations of Grassmann spaces $f':\cG_{k}\to\cG_{k}$
and $f'':\cG_{n-k}\to\cG_{n-k}$ such that $g=f'\times f''$. Clearly, $f'$ and
$f''$ are adjacency preserving in both directions.

Case 2. There exist collineations of Grassmann spaces $g':\cG_{k}\to\cG_{n-k}$
and $g'':\cG_{n-k}\to\cG_{k}$ such that $g=g'\dtimes g''$. As above, $g'$ and
$g''$ are adjacency preserving in both directions.

So $g$ is given as in Example~\ref{exmp:AGxG}.
\end{myproof}

\tiny

\ \\

Hans Havlicek\\
INSTITUT F\"UR DISKRETE MATHEMATIK UND GEOMETRIE\\
TECHNISCHE UNIVERSIT\"AT WIEN\\
WIEDNER HAUPTSTRASSE 8--10\\
A-1040 WIEN, AUSTRIA\\
havlicek@geometrie.tuwien.ac.at\\

\ \\

Mark Pankov\\
INSTITUTE OF MATHEMATICS\\
NATIONAL ACADEMY OF SCIENCE OF UKRAINE\\
TERESHCHENKIVSKA 3\\
01601 KIEV, UKRAINE\\
pankov@imath.kiev.ua

\end{document}